\documentclass[12pt]{article}
\usepackage{amssymb}
\usepackage{amsmath}
\usepackage{amsthm}
\usepackage{color}
\usepackage{comment}
\usepackage{geometry}		
\usepackage{graphicx}

\makeatletter
	\@addtoreset{equation}{section}
\makeatother

\let\originalleft\left
\let\originalright\right
\renewcommand{\left}{\mathopen{}\mathclose\bgroup\originalleft}
\renewcommand{\right}{\aftergroup\egroup\originalright}

\begin{document}

\newcommand{\ee}{\varepsilon}

\newtheorem{theorem}{Theorem}
\newtheorem{conjecture}[theorem]{Conjecture}
\newtheorem{lemma}[theorem]{Lemma}
\newtheorem{proposition}[theorem]{Proposition}

\theoremstyle{definition}
\newtheorem{definition}{Definition}

\title{Nonsmooth folds as tipping points.}

\author{
D.J.W.~Simpson\\\\
School of Mathematical and Computational Sciences\\
Massey University\\
Palmerston North, 4410\\
New Zealand
}
\maketitle

\begin{abstract}

A nonsmooth fold is where an equilibrium or limit cycle of a nonsmooth dynamical system hits a switching manifold and collides and annihilates with another solution of the same type. We show that beyond the bifurcation the leading-order truncation to the system in general has no bounded invariant set. This is proved for boundary equilibrium bifurcations of Filippov systems, hybrid systems, and continuous piecewise-smooth ODEs, and grazing-type events for which the truncated form is a continuous piecewise-linear map. The omitted higher-order terms are expected to be incapable of altering the local dynamics qualitatively, implying the system has no local invariant set on one side of a nonsmooth fold, and we demonstrate this with an example. Thus if the equilibrium or limit cycle is attracting the bifurcation causes the local attractor of the system to tip to a new state. The results also help explain global aspects of the bifurcation structures of the truncated systems. 

\end{abstract}



\section{Introduction}
\label{sec:intro}

Bifurcations are critical parameter values at which the dynamical behaviour of a system changes in a fundamental way.
For example oscillations in chemical kinetics are often onset by Hopf bifurcations \cite{RiRe81},
turbulent fluid flow is usually generated through period-doubling cascades \cite{KaTr92},
and Type I excitability in neurons is caused by SNIC bifurcations \cite{RiEr89}.
These examples require local smoothness in the equations of motion.
If the equations are nonsmooth, additional bifurcations are possible \cite{DiBu08}.

Nonsmooth models commonly belong to one of three classes: piecewise-smooth continuous ODEs, Filippov systems, and hybrid systems.
The phase space of these systems contain switching manifolds where the ODEs are
non-differentiable, discontinuous, or a map is applied, Fig.~\ref{fig:trivAllSchem}.
In each case a {\em boundary equilibrium bifurcation} (BEB) occurs
when an equilibrium of a smooth component of the ODEs hits a switching manifold as parameters are varied.
Functionally the equilibrium is well-defined on both sides of the bifurcation,
but only on one side of the bifurcation is it an equilibrium of the system where it is said to be {\em admissible};
on the other side of the bifurcation it is {\em virtual}.

\begin{figure}[b!]
\begin{center}
\includegraphics[width=10cm]{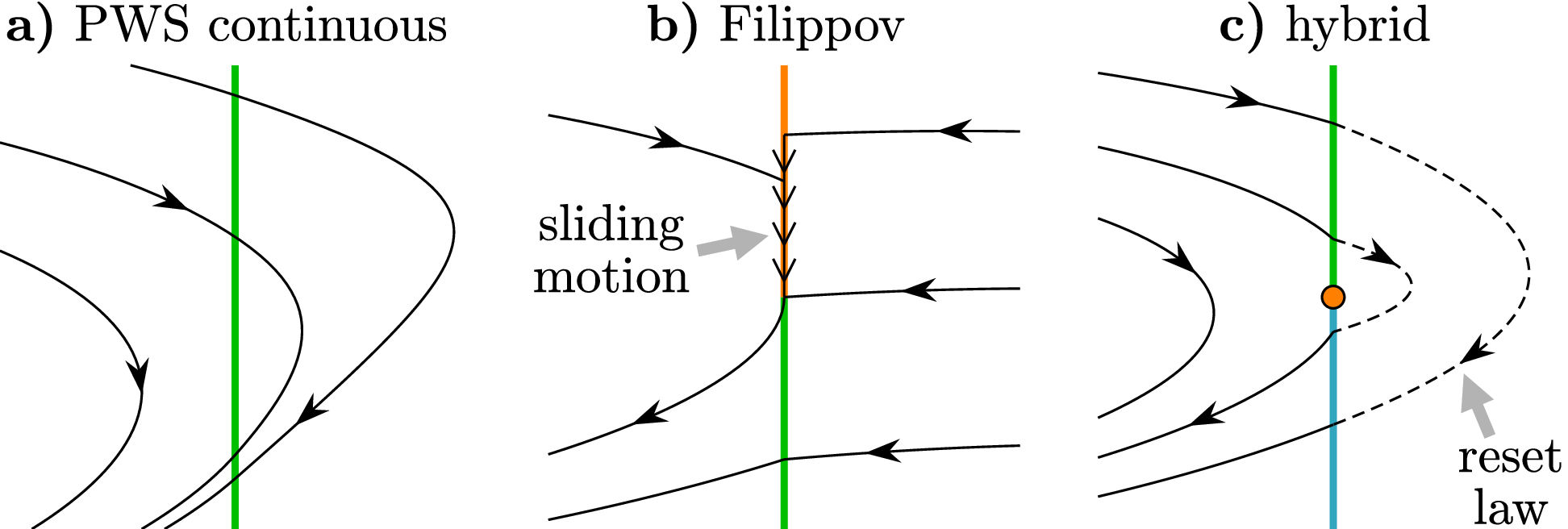}
\caption{
Schematic phase portraits of two-dimensional nonsmooth systems with one switching manifold (vertical line).
In (a) the ODEs are continuous but non-differentiable on the switching manifold.
In (b) the ODEs are discontinuous on the switching manifold
and Filippov's convention is used to specify sliding motion on the switching manifold.
In (c) the system involves ODEs and a map (reset law) that instantaneously transports the system state from the top half of the
switching manifold to the bottom half of the switching manifold.
Such hybrid systems are commonly used to model mechanical systems with hard impacts \cite{BlCz99,Br99,Ib09}.
\label{fig:trivAllSchem}
} 
\end{center}
\end{figure}

Generic BEBs involve two equilibria,
where for Filippov and hybrid systems the second equilibrium is a {\em pseudo-equilibrium}
(an equilibrium of the sliding or sticking motion) \cite{DiNo08}.
Both equilibria are admissible on exactly one side of the bifurcation, and this immediately presents us with two cases:
Either the equilibria are admissible on different sides of the bifurcation, as in Fig.~\ref{fig:trivBEBa},
or the equilibria are admissible on the same side of the bifurcation, as in Fig.~\ref{fig:trivBEBb}.
The first case is termed {\em persistence}.
The second case is termed a {\em nonsmooth fold}
because if we consider only admissible solutions
then at the bifurcation two equilibria collide and annihilate, analogous to a saddle-node bifurcation or fold.

\begin{figure}[b!]
\begin{center}
\includegraphics[width=10cm]{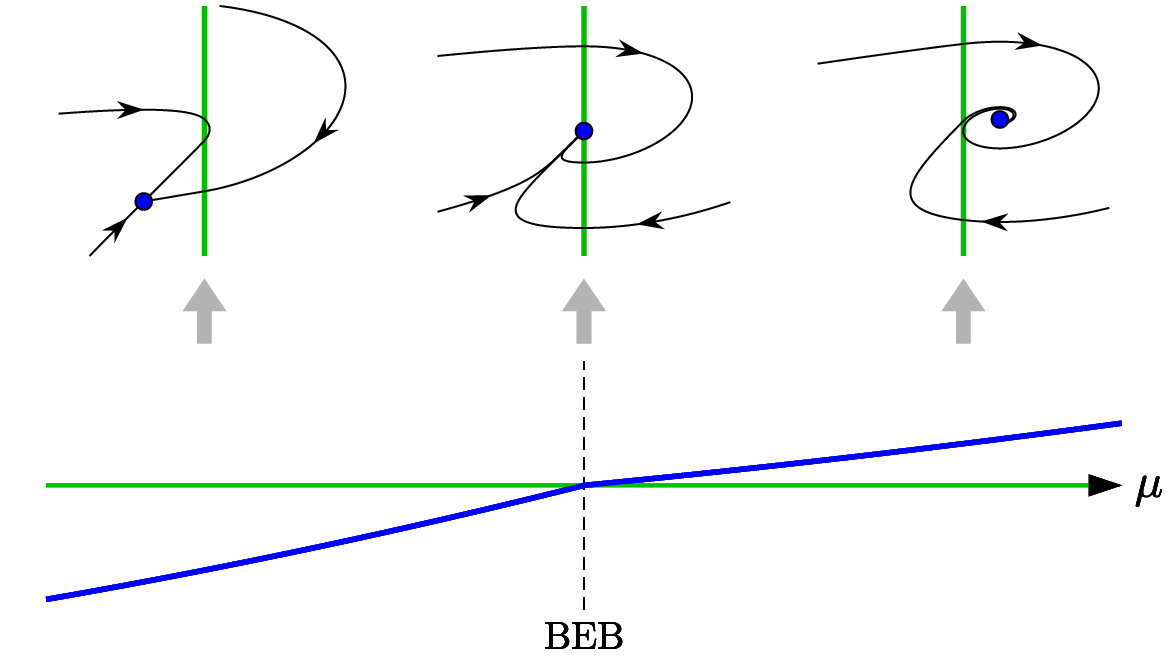}
\caption{
A bifurcation diagram and sample phase portraits of a piecewise-smooth continuous ODE system
that experiences a persistence-type BEB as a parameter $\mu$ is varied.
The blue curves indicate the location of the admissible equilibrium in relation to the switching manifold.
\label{fig:trivBEBa}
} 
\end{center}
\end{figure}

\begin{figure}[b!]
\begin{center}
\includegraphics[width=10cm]{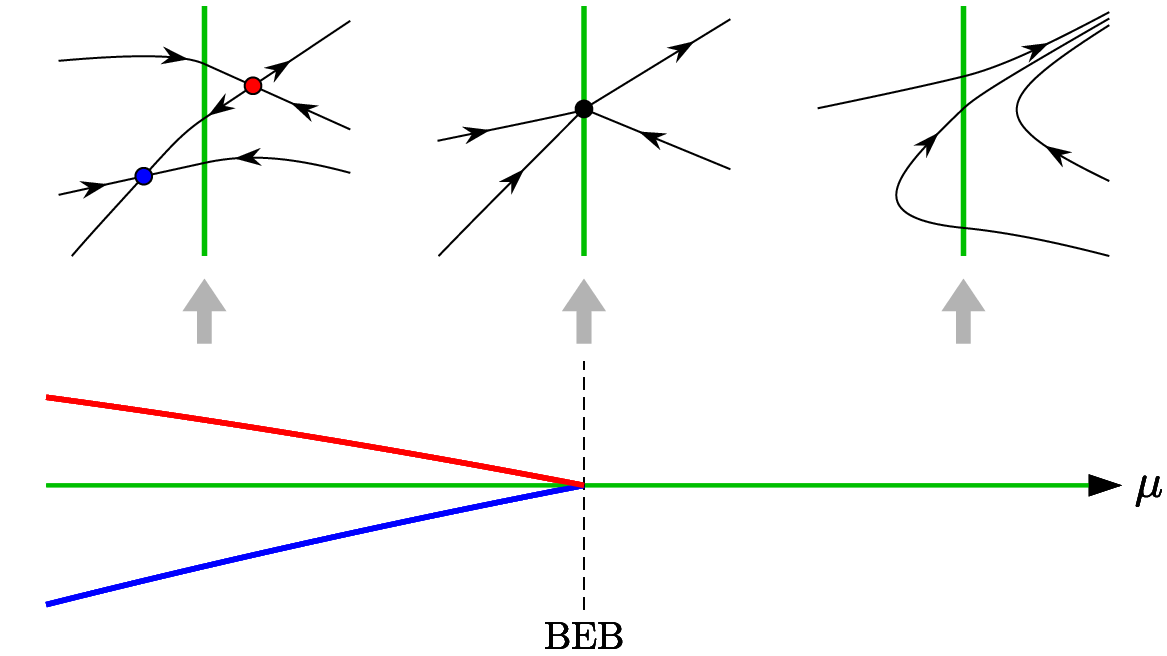}
\caption{
A bifurcation diagram and sample phase portraits of a piecewise-smooth continuous ODE system
that experiences a nonsmooth fold BEB as a parameter $\mu$ is varied.
To the left of the bifurcation the system has two admissible equilibria;
to the right of the bifurcation it has no admissible equilibria.
Throughout this paper stable solutions are coloured blue and unstable solutions are coloured red.
\label{fig:trivBEBb}
} 
\end{center}
\end{figure}

BEBs readily generate other invariant sets growing out of the bifurcation point,
such as limit cycles \cite{KuRi03,SiMe12} and chaotic sets \cite{DiNo08,Si16c,Si18d},
and the various transitions that BEBs bring about have been heavily studied.
It is well known that the local dynamics is usually captured by the leading-order terms of a Taylor expansion
of each smooth component of the system centred at the bifurcation point.
This leads to a truncated form that is more amenable to an exact analysis.
Indeed when analysing families of BEBs
in mathematical models it is usually convenient to work with the truncated form
in order to accurately identify codimension-two points.

In this paper we describe a simple property of BEBs that appears to have been overlooked.
We show that on the side of a nonsmooth fold where both equilibria are virtual
the system has no local invariant sets, at least for the truncated form.
This result applies to piecewise-smooth continuous ODEs, Filippov systems, and hybrid systems,
and with no restriction on the dimension of the system.
Also, if a limit cycle has a degenerate interaction with a switching manifold,
the same result applies when the associated Poincar\'e map is continuous and asymptotically piecewise-linear.
This occurs for grazing-sliding bifurcations \cite{DiKo02}, corner collisions \cite{DiBu01c},
and event collisions for piecewise-smooth ODEs with time-delay in the switching condition \cite{Si06}.
To prove the result, different calculations are required for each setting,
but in each case there is a certain direction in which solutions evolve monotonically and hence eventually diverge.
This is a consequence of piecewise-linearity,
the absence of admissible equilibria,
and a continuity constraint on the vector field or map
(for Filippov systems the vector field can be viewed as continuous through tangency points).

We expect that in all settings if the truncated form has no local invariant set
then neither does the full system because the terms that have been omitted are higher order.
In \S\ref{sec:hots} we verify this for an example and outline the technical difficulties
that need to be overcome in order to prove this in general.
If true then any nonsmooth fold for which one of the two solutions is attracting has the features of a tipping point:
By allowing the parameter that controls the bifurcation to vary slowly in time,
the system state changes rapidly after it passes the nonsmooth fold.
Fig.~\ref{fig:trivStommel} illustrates this for the system
\begin{equation}
\begin{split}
\dot{T} &= 1 - T - k(T,S) T, \\
\dot{S} &= \beta(\mu - S) - k(T,S) S,
\end{split}
\label{eq:oceanModel}
\end{equation}
where $k(T,S) = \alpha \beta |T - S|$.
This is a version of Stommel's two-box model \cite{St61} for ocean circulation
($T$ and $S$ represent temperature and salinity differences).
If $\mu$ is a fixed parameter, the system has a nonsmooth fold at $\mu = 1$.
If $\mu$ allowed to decrease slowly, the system state heads rapidly to the upper equilibrium branch after passing $\mu = 1$.
This has been studied further by Roberts and coworkers \cite{RoGl14,RoSa17} and Budd {\em et al.}~\cite{BuGr22}.

\begin{figure}[b!]
\begin{center}
\includegraphics[width=10cm]{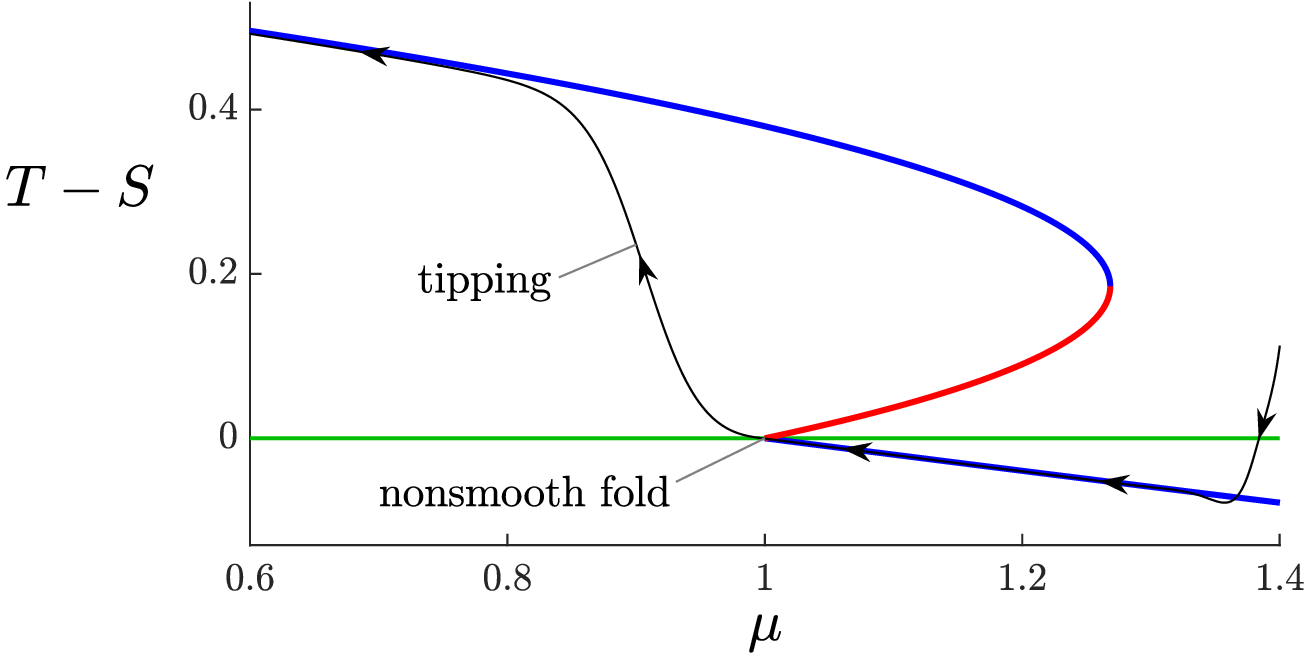}
\caption{
A bifurcation diagram of Stommel's ocean circulation model with $\alpha = 5$ and $\beta = 0.2$.
The blue and red curves are branches of stable and unstable equilibria.
The black solution uses $\dot{\mu} = -0.01$ and experiences a tipping point by passing the nonsmooth fold at $\mu = 1$.
\label{fig:trivStommel}
} 
\end{center}
\end{figure}

It should be stressed that in general the absence of an attractor only occurs locally.
For example in Fig.~\ref{fig:trivBEBc} the nonsmooth fold creates a large amplitude limit cycle.
This is analogous to a SNIC bifurcation for smooth systems. 
Also the occurrence of two virtual stable equilibria
can create stable oscillations \cite{MoBu21,WaWi16,WaWi20}.
For example Fig.~\ref{fig:trivLe18} uses again
\eqref{eq:oceanModel}, but now with $k(T,S) = 1$ if $-\alpha T + S > \ee$, and $k(T,S) = 0$ otherwise
(this is Welander's ocean circulation model \cite{We82} in the discontinuous limit \cite{Le18}).
Orbits move towards one equilibrium until crossing the switching manifold,
then move towards the other equilibrium until crossing the switching manifold, and so on, leading to a stable limit cycle.

\begin{figure}[b!]
\begin{center}
\includegraphics[width=10cm]{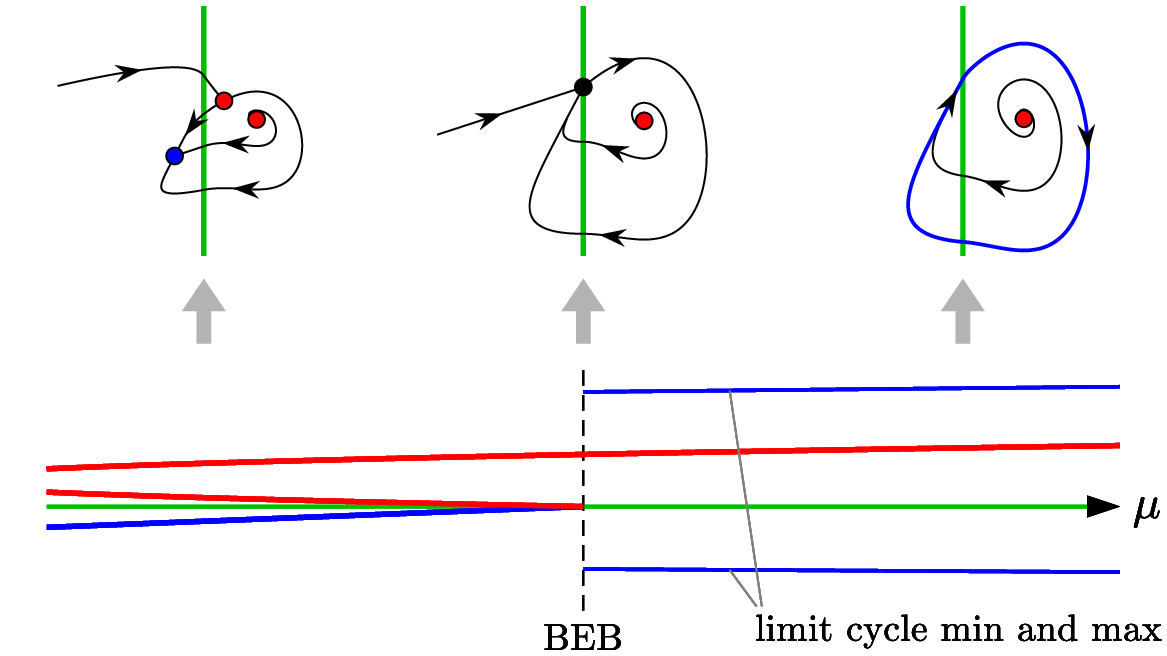}
\caption{
A bifurcation diagram and sample phase portraits of a piecewise-smooth continuous ODE system
that experiences a nonsmooth fold BEB as a parameter $\mu$ is varied.
The bifurcation creates a large amplitude limit cycle due to global features of the dynamics.
\label{fig:trivBEBc}
} 
\end{center}
\end{figure}

\begin{figure}[b!]
\begin{center}
\includegraphics[width=10cm]{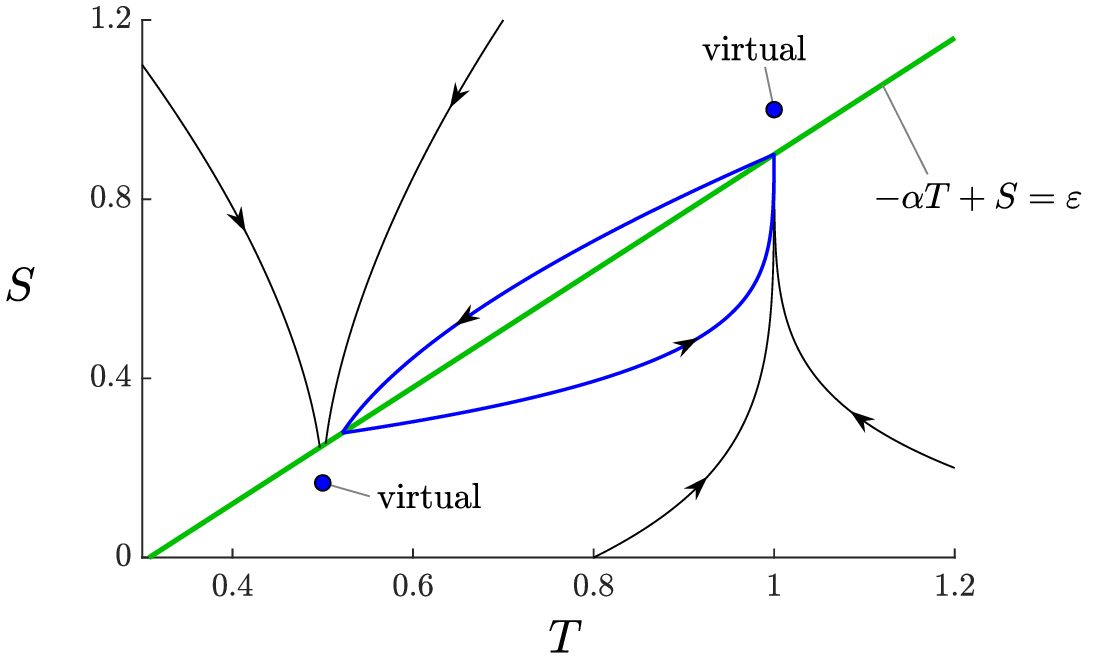}
\caption{
A phase portrait of Welander's ocean circulation model
with $\alpha = 1.3$, $\beta = 0.2$, $\ee = -0.4$, and $\mu = 1$.
There are two virtual equilibria and a stable limit cycle (blue).
\label{fig:trivLe18}
} 
\end{center}
\end{figure}

The results of this paper were inspired by a 2002 paper of Carmona {\em et al.}~\cite{CaFr02}.
Their work provided canonical forms for piecewise-linear ODEs arising in control circuitry.
As a side note (their Proposition 19) they proved that if their two-piece canonical form has no admissible equilibria then it has no periodic solutions.
We generalise their result beyond the canonical form,
from periodic solutions to arbitrary bounded invariant sets,
place the result in the context of BEBs,
and extend the result to Filippov systems, hybrid systems, and maps.

The remainder of the paper is organised as follows.
In \S\ref{sec:adjugate} we describe adjugate matrices
as these are central to our main proofs.
In \S\ref{sec:maps} we prove the result for piecewise-linear maps,
then in \S\ref{sec:hots} explore the influence of higher order terms.

In \S\ref{sec:cont} we treat piecewise-smooth continuous ODEs for which the result follows similarly to the map case.
In \S\ref{sec:Filippov} we treat piecewise-smooth discontinuous ODEs
using Filippov's convention, as is standard \cite{DiBu08,Fi88,Je18b}, to define evolution on switching manifolds.
Then in \S\ref{sec:hybrid} we treat hybrid systems.
Specifically we consider hybrid models of impacting systems where the reset law models impact events.
This is because such models have constraints that mean BEBs occur in a local fashion
and correspond to either persistence or a nonsmooth fold \cite{DiNo08}.
Finally \S\ref{sec:conc} presents a summary and discusses
consequences to global bifurcations of parameterised families of truncated forms.

\section{Adjugate matrices}
\label{sec:adjugate}

Here we define adjugate matrices and briefly describe some of their properties.
Further discussion can be found in standard textbooks \cite{Be92,Ko96}.
Throughout the paper $I$ denotes the $n \times n$ identity matrix
and $e_1$ denotes the first standard basis vector of $\mathbb{R}^n$.

\begin{definition}
The {\em adjugate} of a matrix $A \in \mathbb{R}^{n \times n}$
is the $n \times n$ matrix defined by ${\rm adj}(A)_{ij} = (-1)^{i+j} m_{ji}$,
where $m_{ji}$ is the determinant of the $(n-1) \times (n-1)$ matrix obtained by
removing the $j^{\rm th}$ row and $i^{\rm th}$ column from $A$.
\label{df:adjugate}
\end{definition}

The key property of the adjugate matrix is that for any $A \in \mathbb{R}^{n \times n}$,
\begin{equation}
A \,{\rm adj}(A) = {\rm adj}(A) A = \det(A) I.
\label{eq:adjugateIdentity}
\end{equation}
So if $A$ is invertible,
\begin{equation}
A^{-1} = \frac{{\rm adj}(A)}{\det(A)}.
\label{eq:inverse}
\end{equation}

The first row of ${\rm adj}(A)$ contains the values $(-1)^{i+j} m_{j1}$
which are independent of the entries in the first column of $A$.
Thus if two $n \times n$ matrices differ only in their first columns,
then the first rows of their adjugates will be identical.
Algebraically this means
\begin{equation}
e_1^{\sf T} {\rm adj}(A) = e_1^{\sf T} {\rm adj} \left( A + c e_1^{\sf T} \right)
\label{eq:adjFirstRow}
\end{equation}
for any $A \in \mathbb{R}^{n \times n}$ and $c \in \mathbb{R}^n$.

\section{Maps}
\label{sec:maps}

Let $f$ be a piecewise-$C^1$ continuous map with vector variable $x = (x_1,\ldots,x_n) \in \mathbb{R}^n$
and scalar parameter $\mu \in \mathbb{R}$.
Suppose $f$ has one switching manifold and coordinates
are chosen so that this manifold is $x_1 = 0$.
Then $f$ has the form
\begin{equation}
f(x;\mu) = \begin{cases}
f^L(x;\mu), & x_1 \le 0, \\
f^R(x;\mu), & x_1 \ge 0,
\end{cases}
\label{eq:mf}
\end{equation}
where $f^L$ and $f^R$ are $C^1$.

As $\mu$ is varied a border-collision bifurcation occurs when a
fixed point of $f^L$ or $f^R$ hits the switching manifold.
If $f$ is a Poincar\'e map of an piecewise-smooth ODE system,
its fixed points correspond to periodic orbits and the border-collision bifurcation corresponds to
a periodic orbit that interacts degenerately with a switching manifold of the ODEs \cite{DiBu08}.
Examples of this include grazing-sliding bifurcations \cite{DiKo02}, corner collisions \cite{DiBu01c}, and event collisions \cite{Si06}.

Suppose $f$ has a border-collision bifurcation at $x = 0$ when $\mu = 0$.
Then $f^L(0;0) = f^R(0;0) = 0$, so we can write
\begin{equation}
f(x;\mu) = \begin{cases}
A_L x + b \mu + E^L(x;\mu), & x_1 \le 0, \\
A_R x + b \mu + E^R(x;\mu), & x_1 \ge 0,
\end{cases}
\label{eq:mf2}
\end{equation}
where $A_L, A_R \in \mathbb{R}^{n \times n}$ differ only in their first columns (by continuity),
$b \in \mathbb{R}^n$, and $E^L(x;\mu)$ and $E^R(x;\mu)$ contain only higher order terms.
By dropping $E^L$ and $E^R$ we obtain the truncated form
\begin{equation}
g(x;\mu) = \begin{cases}
A_L x + b \mu, & x_1 \le 0, \\
A_R x + b \mu, & x_1 \ge 0,
\end{cases}
\label{eq:mg}
\end{equation}
which approximates $f$ in a neighbourhood of $(x;\mu) = (0;0)$.
In the remainder of this section we study the truncated form $g$;
the general map $f$ is revisited in the next section.

If $\det(I - A_L) \ne 0$ and $\det(I - A_R) \ne 0$, as is usually the case,
the pieces of $g$ have the unique fixed points
\begin{equation}
\begin{split}
x^L(\mu) &= (I - A_L)^{-1} b \mu, \\
x^R(\mu) &= (I - A_R)^{-1} b \mu.
\end{split}
\label{eq:mxLxR}
\end{equation}
The point $x^L(\mu)$ is admissible
if $x^L_1(\mu) < 0$ and virtual if $x^L_1(\mu) > 0$.
Similarly $x^R(\mu)$ is admissible if $x^R_1(\mu) > 0$ and virtual if $x^R_1(\mu) < 0$.
We can now state the main result for maps.

\begin{theorem}
If $\det(I - A_L) \ne 0$, $\det(I - A_R) \ne 0$, and both fixed points of \eqref{eq:mg} are virtual,
then every forward orbit of \eqref{eq:mg} diverges.
\label{th:m}
\end{theorem}

Below we prove Theorem \ref{th:m} by showing there exists $s > 0$ and $p \in \mathbb{R}^n$ such that
\begin{equation}
p^{\sf T} g(x;\mu) \ge p^{\sf T} x + s,
\label{eq:mInequality}
\end{equation}
for all $x \in \mathbb{R}^n$.
This shows that the value of $p^{\sf T} x$ increases without bound along forward orbits of $g$,
hence all forward orbits diverge.

As an example for which this can be seen graphically, consider \eqref{eq:mg} with
\begin{equation}
\begin{split}
A_L &= \begin{bmatrix} \delta_L + 1 - \alpha & 1 \\ -\delta_L & 0 \end{bmatrix}, \\
A_R &= \begin{bmatrix} \delta_R + 1 + \alpha & 1 \\ -\delta_R & 0 \end{bmatrix}, \\
b &= \begin{bmatrix} 1 \\ 0 \end{bmatrix},
\end{split}
\label{eq:example}
\end{equation}
where $\delta_L, \delta_R \in \mathbb{R}$ and $\alpha > 0$ are additional parameters.
The fixed points are
\begin{equation}
\begin{split}
x^L(\mu) &= \left( \frac{\mu}{\alpha}, \frac{-\delta_L \mu}{\alpha} \right), \\
x^R(\mu) &= \left( \frac{-\mu}{\alpha}, \frac{\delta_R \mu}{\alpha} \right),
\end{split}
\nonumber
\end{equation}
so are both virtual when $\mu > 0$.
The inequality \eqref{eq:mInequality} holds for all $x \in \mathbb{R}^n$ using
$s = \mu$ and $p^{\sf T} = \begin{bmatrix} 1 & 1 \end{bmatrix}$
(obtained by applying the formulas below to this example).
Fig.~\ref{fig:trivqqDivergence} uses $\mu = 1$
and shades regions between integer values of $p^{\sf T} x$.
The inequality \eqref{eq:mInequality} implies that no two points of an orbit of $g$
belong to the same region, and this is evident for the sample orbit shown.

\begin{figure}[b!]
\begin{center}
\includegraphics[width=10cm]{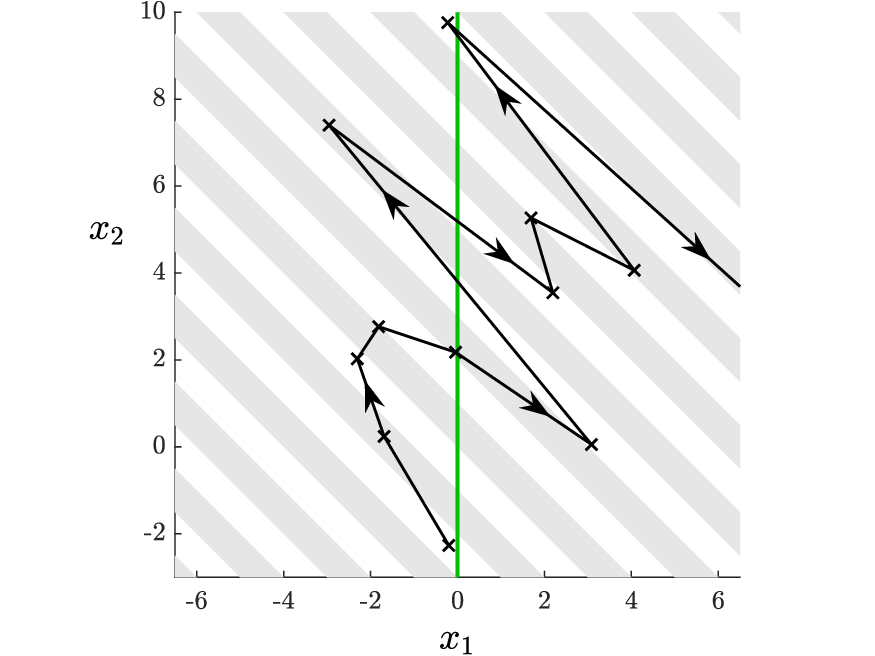}
\caption{
A forward orbit of the piecewise-linear map \eqref{eq:mg}
with \eqref{eq:example} and $\delta_L = 1.2$, $\delta_R = -2.4$, $\alpha = 0.1$, and $\mu = 1$.
The shaded regions are bounded by integer values of $p^{\sf T} x = x_1 + x_2$.
\label{fig:trivqqDivergence}
} 
\end{center}
\end{figure}

\begin{proof}[Proof of Theorem \ref{th:m}]
By \eqref{eq:inverse} applied to \eqref{eq:mxLxR},
\begin{equation}
x^L_1(\mu) = \frac{p^{\sf T} b \mu}{\det(I - A_L)},
\label{eq:mxL1}
\end{equation}
where we let
\begin{equation}
p^{\sf T} = e_1^{\sf T} {\rm adj}(I - A_L)
\nonumber
\end{equation}
denote the first row of ${\rm adj}(I - A_L)$.
By \eqref{eq:adjFirstRow} this is also the first row of ${\rm adj}(I - A_R)$, so
\begin{equation}
x^R_1(\mu) = \frac{p^{\sf T} b \mu}{\det(I - A_R)}.
\label{eq:mxR1}
\end{equation}
Let
\begin{equation}
s = p^{\sf T} b \mu.
\nonumber
\end{equation}
Notice $s \ne 0$, for otherwise the fixed points are not virtual.
Without loss of generality suppose $s > 0$.
Then $\det(I - A_L) > 0$ and $\det(I - A_R) < 0$ because the fixed points are virtual.

For any $x \in \mathbb{R}^n$ with $x_1 \le 0$,
\begin{align*}
p^{\sf T} g(x;\mu)
&= p^{\sf T} (A_L x + b \mu) \\
&= p^{\sf T} (-x + A_L x + x + b \mu) \\
&= -p^{\sf T} (I - A_L) x + p^{\sf T} x + s.
\end{align*}
By \eqref{eq:adjugateIdentity} we have $p^{\sf T} (I - A_L) x = \det(I - A_L) x_1$.
Importantly $\det(I - A_L) > 0$ and $x_1 \le 0$, hence
\begin{equation}
p^{\sf T} g(x;\mu) \ge p^{\sf T} x + s.
\label{eq:mIncrease}
\end{equation}
Similarly for any $x \in \mathbb{R}^n$ with $x_1 \ge 0$,
\begin{equation}
p^{\sf T} g(x;\mu) = -p^{\sf T} (I - A_R) x + p^{\sf T} x + s,
\nonumber
\end{equation}
where $p^{\sf T} (I - A_R) x = \det(I - A_R) x_1 \le 0$ because $\det(I - A_R) < 0$,
so again we have \eqref{eq:mIncrease}.
Thus for any forward orbit of $g$,
the value of $p^{\sf T} x$ increases by at least $s > 0$ every iteration, so the orbit diverges.
\end{proof}

\section{The effect of higher order terms}
\label{sec:hots}

The quantities $E^L$ and $E^R$ in \eqref{eq:mf2} that were omitted to produce the truncated form \eqref{eq:mg} are higher order.
To be precise, $\frac{\| E^L(x;\mu) \|}{\| x \| + |\mu|} \to 0$
and $\frac{\| E^R(x;\mu) \|}{\| x \| + |\mu|} \to 0$ as $(x;\mu) \to (0;0)$,
where $\| \cdot \|$ denotes the Euclidean norm on $\mathbb{R}^n$.
For this reason if \eqref{eq:mg} satisfies the conditions of Theorem \ref{th:m},
we expect \eqref{eq:mf2} has no local invariant set on one side of the nonsmooth fold.
The following conjecture formalises this claim.
Without loss of generality we suppose $p^{\sf T} b > 0$ and consider $\mu > 0$
in which case $\det(I - A_L) > 0$ and $\det(I - A_R) < 0$ are needed
for both fixed points to be virtual by \eqref{eq:mxL1} and \eqref{eq:mxR1}.
We write $B_\eta = \left\{ x \in \mathbb{R}^n \,\big|\, \| x \| \le \eta \right\}$
for the closed ball of radius $\eta > 0$ centred at the origin.

\begin{conjecture}
Let $f$ be a map of the form \eqref{eq:mf2} with $\det(I - A_L) > 0$, $\det(I - A_R) < 0$, and $p^{\sf T} b > 0$.
Then there exists $\mu_0 > 0$ and $\eta > 0$ such that
for all $\mu \in (0,\mu_0)$ and all $x \in B_\eta$
there exists $m \ge 1$ such that $\| f^m(x;\mu) \| > \eta$.
\label{cj:hots}
\end{conjecture}

That is, $f$ has no invariant set in $B_\eta$ for all $\mu \in (0,\mu_0)$.
To illustrate Conjecture \ref{cj:hots} we add the quadratic terms
$E^L(x;\mu) = E^R(x;\mu) = \begin{bmatrix} 0 \\ -x_2^2 \end{bmatrix}$
to our previous example of \eqref{eq:mg} with \eqref{eq:example} to obtain the map
\begin{equation}
f(x;\mu) =
\begin{cases}
\begin{bmatrix} (\delta_L + 1 - \alpha) x_1 + x_2 + \mu \\ -\delta_L x_1 - x_2^2 \end{bmatrix}, & x_1 \le 0, \\[4mm]
\begin{bmatrix} (\delta_R + 1 + \alpha) x_1 + x_2 + \mu \\ -\delta_R x_1 - x_2^2 \end{bmatrix}, & x_1 \ge 0.
\end{cases}
\label{eq:mapExampleHOTs}
\end{equation}
Conjecture \ref{cj:hots} is true for this map with the parameter values of Fig.~\ref{fig:trivqqDivergence},
but to prove this it is insufficient to use the linear function $p^{\sf T} x = x_1 + x_2$.
This is because
\begin{equation}
p^{\sf T} f(x;\mu) - p^{\sf T} x = \begin{cases}
\mu - \alpha x_1 - x_2^2 \,, & x_1 \le 0, \\
\mu + \alpha x_1 - x_2^2 \,, & x_1 \ge 0,
\end{cases}
\nonumber
\end{equation}
so for any $\eta > 0$ and $\mu \in \left( 0, \eta^2 \right)$,
at the point $x = (0,\eta) \in B_\eta$ the difference
$p^{\sf T} f(x;\mu) - p^{\sf T} x = \mu - \eta^2$ is negative.
To prove Conjecture \ref{cj:hots} for \eqref{eq:mapExampleHOTs} 
we can instead use $\Phi(x) = x_1 + x_2 - 2 x_2^2$, as shown in Appendix \ref{app:hots}.
It remains for future work to show that such a function $\Phi$ can always be constructed,
prove Conjecture \ref{cj:hots} in different way,
or find a counter-example.

\section{Piecewise-smooth continuous ODEs}
\label{sec:cont}

We now consider BEBs of piecewise-continuous ODEs.
Our main result is a continuous-time analogue of Theorem \ref{th:m}.

Consider $n$-dimensional ODE systems of the form
\begin{equation}
\dot{x} = \begin{cases}
f^L(x;\mu), & x_1 \le 0, \\
f^R(x;\mu), & x_1 \ge 0,
\end{cases}
\label{eq:cfGen}
\end{equation}
where $f^L$ and $f^R$ are $C^1$ and $f^L(x;\mu) = f^R(x;\mu)$
at all points on the switching manifold $x_1 = 0$.
Suppose \eqref{eq:cfGen} has a BEB at $x = 0$ when $\mu = 0$.
Then by replacing $f^L$ and $f^R$ with their linearisations we obtain
\begin{equation}
\dot{x} = \begin{cases}
A_L x + b \mu, & x_1 \le 0, \\
A_R x + b \mu, & x_1 \ge 0,
\end{cases}
\label{eq:cf}
\end{equation}
where $A_L, A_R \in \mathbb{R}^{n \times n}$ differ in only their first columns (by continuity)
and $b \in \mathbb{R}^n$.
Solutions to \eqref{eq:cf} are well-defined for all $x \in \mathbb{R}^n$ and $t \in \mathbb{R}$
because \eqref{eq:cf} is Lipschitz and each piece is linear.

If $\det(A_L) \ne 0$ and $\det(A_R) \ne 0$ the pieces of \eqref{eq:cf} have the unique equilibria
\begin{equation}
\begin{split}
x^L(\mu) &= -A_L^{-1} b \mu, \\
x^R(\mu) &= -A_R^{-1} b \mu.
\label{eq:cxLxR}
\end{split}
\end{equation}
The equilibrium $x^L(\mu)$ is admissible if $x^L_1(\mu) < 0$ and virtual if $x^L_1(\mu) > 0$,
while $x^R(\mu)$ is admissible if $x^R_1(\mu) > 0$ and virtual if $x^R_1(\mu) < 0$.

\begin{theorem}
If $\det(A_L) \ne 0$, $\det(A_R) \ne 0$, and both equilibria of \eqref{eq:cf} are virtual,
then every solution to \eqref{eq:cf} diverges as $t \to \pm \infty$.
\label{th:c}
\end{theorem}

\begin{proof}
By \eqref{eq:inverse} applied to \eqref{eq:cxLxR},
\begin{equation}
x^L_1(\mu) = -\frac{q^{\sf T} b \mu}{\det(A_L)},
\label{eq:cxL1}
\end{equation}
where we let
\begin{equation}
q^{\sf T} = e_1^{\sf T} {\rm adj}(A_L).
\nonumber
\end{equation}
By \eqref{eq:adjFirstRow} this is also the first row of ${\rm adj}(A_R)$, so
\begin{equation}
x^R_1(\mu) = -\frac{q^{\sf T} b \mu}{\det(A_R)}.
\label{eq:cxR1}
\end{equation}
The value
\begin{equation}
s = q^{\sf T} b \mu
\nonumber
\end{equation}
cannot be zero because the equilibria are virtual.
Without loss of generality suppose $s > 0$;
then $\det(A_L) < 0$ and $\det(A_R) > 0$.

For any $x \in \mathbb{R}^n$ with $x_1 \le 0$,
\begin{equation}
q^{\sf T} f^L(x;\mu)
= q^{\sf T} (A_L x + b \mu)
= q^{\sf T} A_L x + s.
\nonumber
\end{equation}
Notice $q^{\sf T} A_L x = \det(A_L) x_1 \ge 0$, thus
\begin{equation}
q^{\sf T} \dot{x} \ge s.
\label{eq:cIncrease}
\end{equation}
Similarly for any $x \in \mathbb{R}^n$ with $x_1 \ge 0$,
\begin{equation}
q^{\sf T} f^R(x;\mu)
= q^{\sf T} (A_R x + b \mu)
= q^{\sf T} A_R x + s,
\nonumber
\end{equation}
where $q^{\sf T} A_R x = \det(A_R) x_1 \ge 0$, so again we have \eqref{eq:cIncrease}.
Thus any solution $\phi(t)$ to \eqref{eq:cf} has
$\frac{d}{d t} q^{\sf T} \phi(t) \ge s > 0$ for all $t \in \mathbb{R}$,
and so $\| \phi(t) \| \to \infty$ as $t \to \pm \infty$.
\end{proof}

\section{Filippov systems}
\label{sec:Filippov}

Here we treat systems
\begin{equation}
\dot{x} = \begin{cases}
f^L(x;\mu), & x_1 < 0, \\
f^R(x;\mu), & x_1 > 0,
\end{cases}
\label{eq:ffGen}
\end{equation}
where again $f^L$ and $f^R$ are $C^1$,
but now \eqref{eq:ffGen} is not necessarily continuous on
the switching manifold $\Sigma = \left\{ x \in \mathbb{R}^n \,\middle|\, x_1 = 0 \right\}$.
To describe the dynamics of \eqref{eq:ffGen}
it is useful to partition $\Sigma$
into regions throughout which $f^L$ and $f^R$ each either direct solutions into $\Sigma$ or away from $\Sigma$:

\begin{definition}
Consider a system \eqref{eq:ffGen} with fixed $\mu \in \mathbb{R}$.
A subset $S \subset \Sigma$ is a
\begin{enumerate}
\item
a {\em crossing region} if $f^L_1(x;\mu) f^R_1(x;\mu) > 0$ for all $x \in S$, and
\item
a {\em sliding region} if $f^L_1(x;\mu) f^R_1(x;\mu) < 0$ for all $x \in S$.
\end{enumerate}
A sliding region is {\em attracting} if $f^L_1(x;\mu) > 0$ and $f^R_1(x;\mu) < 0$,
and {\em repelling} if $f^L_1(x;\mu) < 0$ and $f^R_1(x;\mu) > 0$.
\label{df:fRegions}
\end{definition}

For example in Fig.~\ref{fig:trivFilippovSchem} the top half
of $\Sigma$ is an attracting sliding region,
while the bottom half of $\Sigma$ is a crossing region.

\begin{figure}[b!]
\begin{center}
\includegraphics[width=10cm]{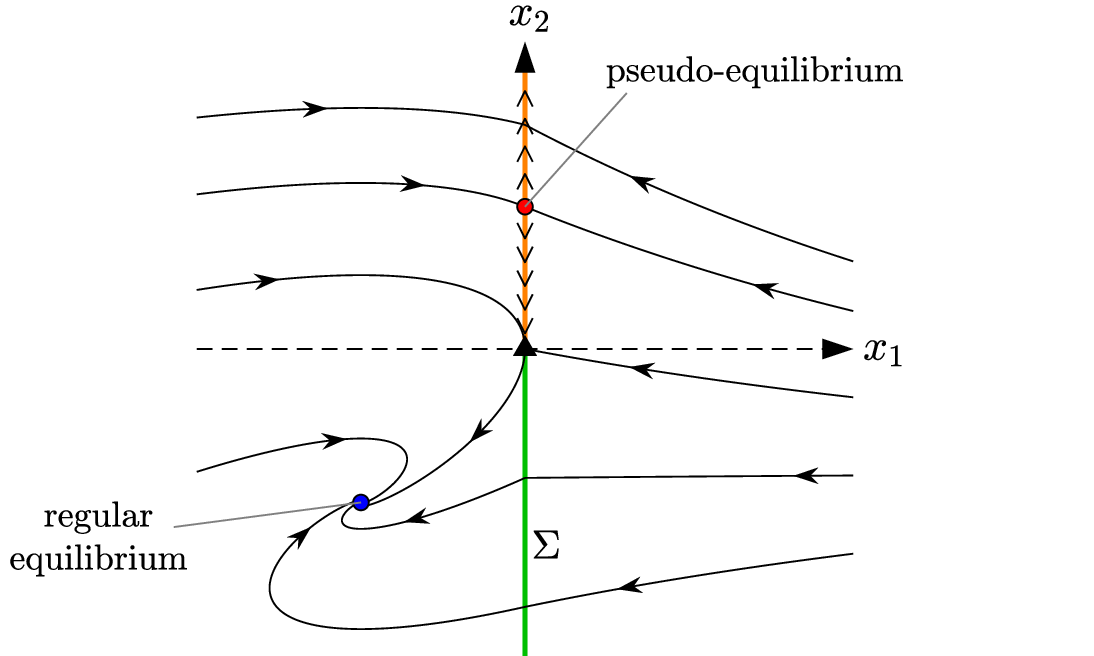}
\caption{
A schematic phase portrait of a two-dimensional Filippov system of the form \eqref{eq:ffGen}.
\label{fig:trivFilippovSchem}
} 
\end{center}
\end{figure}

To define solutions on sliding regions we use Filippov's convention \cite{Fi60,Fi88}.
That is, solutions on sliding regions evolve according to $\dot{x} = f^S(x;\mu)$,
where $f^S$ is the {\em sliding vector field} 
\begin{equation}
f^S(x;\mu) = \frac{f^L_1(x;\mu) f^R(x;\mu) - f^R_1(x;\mu) f^L(x;\mu)}{f^L_1(x;\mu) - f^R_1(x;\mu)},
\label{eq:ffS}
\end{equation}
defined as the unique convex combination of $f^L$ and $f^R$ that is tangent to $\Sigma$.
With this convention \eqref{eq:ffGen} is termed a {\em Filippov system}.
Solutions are termed {\em Filippov solutions} and are concatenations of
smooth segments of motion in $x_1 < 0$ under $f^L$,
in $x_1 > 0$ under $f^R$,
and on sliding regions under $f^S$.

Equilibria of $f^L$ and $f^R$ are admissible or virtual as in the previous section.
They are sometimes called {\em regular equilibria} to distinguish them from equilibria of $f^S$:

\begin{definition}
A point $x \in \Sigma$ is a {\em pseudo-equilibrium} of \eqref{eq:ffGen} if $f^S(x;\mu) = 0$.
It is {\em admissible} if it belongs to a sliding region, and {\em virtual} if it belongs to a crossing region.
\label{df:fPsEq}
\end{definition}

Now suppose \eqref{eq:ffGen} has a BEB
caused by an equilibrium of $f^L$ hitting $\Sigma$ at $x = 0$ when $\mu = 0$.
Then $f^L(0;0) = 0$, but unlike in the previous section typically $f^R(0;0) \ne 0$.
By replacing $f^L(x;\mu)$ with its linearisation about $(x;\mu) = (0;0)$
and $f^R(x;\mu)$ with its value at $(x;\mu) = (0;0)$, we obtain the truncated form
\begin{equation}
\dot{x} = \begin{cases}
A x + b \mu, & x_1 < 0, \\
c, & x_1 > 0,
\end{cases}
\label{eq:ff}
\end{equation}
where $A \in \mathbb{R}^{n \times n}$ and $b, c \in \mathbb{R}^n$.
If $\det(A) \ne 0$ then $f^L$ has the unique equilibrium
\begin{equation}
x^L(\mu) = -A^{-1} b \mu,
\label{eq:fxL}
\end{equation}
which is admissible if $x^L_1 < 0$ and virtual if $x^L_1 > 0$.
If $c_1 \ne 0$ and $q^{\sf T} c \ne 0$, where again
\begin{equation}
q^{\sf T} = e_1^{\sf T} {\rm adj}(A),
\nonumber
\end{equation}
then \eqref{eq:ff} has a unique pseudo-equilibrium $x^S(\mu)$ with $x^S(0) = 0$ \cite{DiNo08,Si18d}.

\begin{theorem}
If $\det(A) \ne 0$, $c_1 \ne 0$, $q^{\sf T} c \ne 0$,
and $x^L(\mu)$ and $x^S(\mu)$ are virtual, then every Filippov solution to \eqref{eq:ff} diverges as $t \to \infty$.
\label{th:f}
\end{theorem}

Let us first give some intuition behind Theorem \ref{th:f}.
As in the previous proof, below we obtain $q^{\sf T} f^L(x;\mu) > 0$ for points with $x_1 \le 0$;
but instead of $f^R$ we work with $f^S$ because the second equilibrium is a zero of $f^S$ not $f^R$.
Filippov solutions switch from following $f^S$ to $f^L$ by passing through points $x \in \Sigma$ for which $f^L_1(x;\mu) = 0$,
and notice $f^S(x;\mu) = f^L(x;\mu)$ at such points by \eqref{eq:ffS}.
This is a type of continuity constraint that 
together with the assumption that $x^S(\mu)$ is virtual,
leads to the inequality $q^{\sf T} f^S(x;\mu) > 0$ throughout sliding regions
and the conclusion that all Filippov solutions diverge.

\begin{proof}
By \eqref{eq:inverse} applied to \eqref{eq:fxL},
\begin{equation}
x^L_1(\mu) = -\frac{q^{\sf T} b \mu}{\det(A)}.
\label{eq:fxL1}
\end{equation}
Thus the admissibility of $x^L(\mu)$ is determined by the signs of $\det(A)$ and
\begin{equation}
s = q^{\sf T} b \mu.
\nonumber
\end{equation}
Notice $s \ne 0$ because $x^L(\mu)$ is assumed to be virtual.

We now derive an analogous characterisation for the admissibility of $x^S(\mu)$.
Since $x^S(\mu)$ is zero of $f^S$,
\begin{equation}
f^L_1(x^S(\mu);\mu) f^R(x^S(\mu);\mu) - f^R_1(x^S(\mu);\mu) f^L(x^S(\mu);\mu) = 0.
\nonumber
\end{equation}
By multiplying both sides of this equation by $q^{\sf T}$, then solving for $f^L_1$, we obtain
\begin{equation}
f^L_1(x^S(\mu);\mu) = \frac{f^R_1(x^S(\mu);\mu) q^{\sf T} f^L(x^S(\mu);\mu)}{q^{\sf T} f^R(x^S(\mu);\mu)}.
\label{eq:fAdmProof}
\end{equation}
Notice $f_R(x^S(\mu);\mu) = c$ and
\begin{align*}
q^{\sf T} f^L(x^S(\mu);\mu)
&= q^{\sf T} \left( A x^S(\mu) + b \mu \right) \\
&= \det(A) x^S_1(\mu) + s \\
&= s,
\end{align*}
where $x^S_1(\mu) = 0$ because $x^S(\mu) \in \Sigma$.
Inserting these into \eqref{eq:fAdmProof} produces
\begin{equation}
f^L_1(x^S(\mu);\mu) = \frac{s c_1}{q^{\sf T} c}.
\nonumber
\end{equation}
Thus the admissibility of $x^S(\mu)$ is determined by the sign of
\begin{equation}
f^L_1(x^S(\mu);\mu) f^R_1(x^S(\mu);\mu) = \frac{s c_1^2}{q^{\sf T} c}.
\label{eq:fAdm}
\end{equation}

Without loss of generality suppose $s > 0$.
Then $\det(A) < 0$, because $x^L(\mu)$ is virtual,
and $q^{\sf T} c > 0$, because $x^S(\mu)$ is virtual.
For any $x \in \mathbb{R}^n$ with $x_1 \le 0$,
\begin{equation}
q^{\sf T} f^L(x;\mu) \ge s,
\nonumber
\end{equation}
as in the previous section.
For any $x \in \Sigma$ for which the denominator of \eqref{eq:ffS} is nonzero,
\begin{align*}
q^{\sf T} f^S(x;\mu)
&= \frac{f^L_1(x;\mu) q^{\sf T} c - c_1 q^{\sf T} f^L(x;\mu)}{f^L_1(x;\mu) - c_1} \\
&= \frac{s + r(x;\mu) q^{\sf T} c}{1 + r(x;\mu)},
\end{align*}
where $q^{\sf T} f^L(x;\mu) = s$ because $x_1 = 0$,
and we have let $r(x;\mu) = -\frac{f^L_1(x;\mu)}{c_1}$.
If $f^L_1(x;\mu) c_1 \le 0$, i.e.~$x$ does not belong to a crossing region, then $r(x;\mu) \ge 0$.
Notice $q^{\sf T} f^S(x;\mu)$ ranges monotonically from $s > 0$ at $r = 0$ to $q^{\sf T} c > 0$ as $r \to \infty$.
Thus at all $x \in \Sigma$ not on a crossing region
the denominator of \eqref{eq:ffS} is nonzero and
\begin{equation}
q^{\sf T} f^S(x;\mu) \ge {\rm min} \left( s, q^{\sf T} c \right) > 0.
\nonumber
\end{equation}
Finally consider the behaviour of an arbitrary Filippov solution $\phi(t)$ to \eqref{eq:ff} as $t \to \infty$.
If $c_1 < 0$ then $\phi(t)$ eventually evolves exclusively under $f^L$ and $f^S$ and diverges,
while if $c_1 > 0$ then $\phi(t)$ either evolves exclusively under $f^L$ and $f^S$ and diverges,
or eventually evolves exclusively under $f^R$ and diverges.
\end{proof}

\section{Impacting hybrid systems}
\label{sec:hybrid}

Finally we consider hybrid systems of the form
\begin{equation}
\begin{split}
\dot{x} &= f(x;\mu), \quad \text{for $x_1 < 0$}, \\
x &\mapsto g(x;\mu), \quad \text{whenever $x_1 = 0$},
\end{split}
\label{eq:hfGen}
\end{equation}
where $f$ and $g$ are $C^1$.
The following assumptions are motivated by the view
that $x_1(t) < 0$ represents the displacement of a rigid object
relative to a wall located at $x_1 = 0$.
The map $g$ is a reset law that models impacts as instantaneous events with velocity reversal and possibly energy loss.

Let $\Sigma = \left\{ x \in \mathbb{R}^n \,\middle|\, x_1 = 0 \right\}$ denote the switching manifold, and
\begin{equation}
v(x;\mu) = f_1(x;\mu)
\label{eq:hvGen}
\end{equation}
denote the velocity of the object relative to the wall.
We partition $\Sigma$ into the {\em incoming set} $\Sigma_{\rm in} = \left\{ x \in \Sigma \,\middle|\, v(x;\mu) > 0 \right\}$,
the {\em outgoing set} $\Sigma_{\rm out} = \left\{ x \in \Sigma \,\middle|\, v(x;\mu) < 0 \right\}$,
and the {\em grazing set} $\Gamma = \left\{ x \in \Sigma \,\middle|\, v(x;\mu) = 0 \right\}$, Fig.~\ref{fig:trivHybridSchem}.
We assume $g$ maps $\Sigma_{\rm in}$ to $\Sigma_{\rm out}$
and is the identity map on $\Gamma$, thus
\begin{equation}
g(x;\mu) = x + v(x;\mu) h(x;\mu),
\label{eq:hg}
\end{equation}
for some $C^1$ function $h$.
The following definition uses
\begin{equation}
a(x;\mu) = \nabla v(x;\mu)^{\sf T} f(x;\mu),
\label{eq:haGen}
\end{equation}
which represents the acceleration of the object relative to the switching manifold.

\begin{figure}[b!]
\begin{center}
\includegraphics[width=10cm]{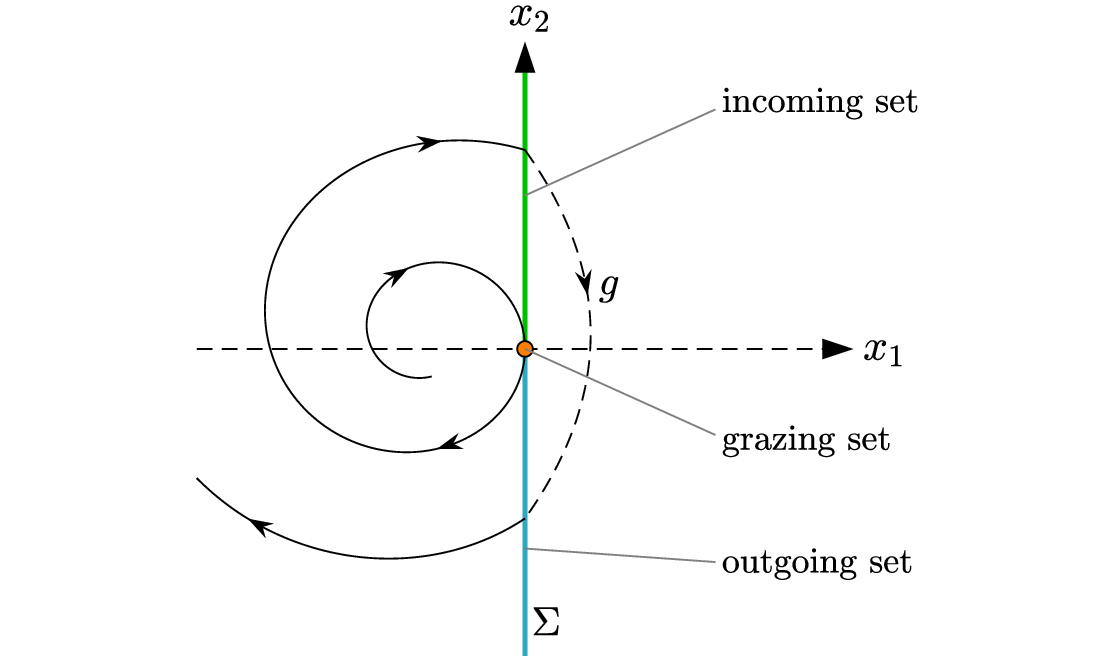}
\caption{
A schematic phase portrait of a two-dimensional hybrid system of the form \eqref{eq:hfGen}.
\label{fig:trivHybridSchem}
} 
\end{center}
\end{figure}

\begin{definition}
Consider a system \eqref{eq:hfGen} with fixed $\mu \in \mathbb{R}$.
A subset $S \subset \Gamma$ is a
\begin{enumerate}
\item
a {\em sticking region} if $a(x;\mu) > 0$ for all $x \in S$, and
\item
a {\em detaching region} if $a(x;\mu) < 0$ for all $x \in S$.
\end{enumerate}
\label{df:hRegions}
\end{definition}

Sticking regions and detaching regions are analogous to sliding regions
and crossing regions of Filippov systems.
On sticking regions it is natural to define the {\em sticking vector field}
\begin{equation}
f^{St}(x;\mu) = \left( I - \frac{h(x;\mu) \nabla v(x;\mu)^{\sf T}}{\nabla v(x;\mu)^{\sf T} h(x;\mu)} \right) f(x;\mu),
\label{eq:hfStGen}
\end{equation}
given by projecting $f$ onto $\Gamma$ by translating in the direction $h$
(see Section 2.2.4 of di Bernardo {\em et al.}~\cite{DiBu08}).
Solutions to \eqref{eq:hfGen} are concatenations of smooth segments of motion in $x_1 < 0$ under $f$,
instantaneous jumps defined by $g$,
and smooth segments of motion on sticking regions under $f^{St}$.
Regular equilibria of \eqref{eq:hfGen} are zeros of $f$;
pseudo-equilibria of \eqref{eq:hfGen} are zeros of $f^{St}$:

\begin{definition}
A point $x \in \Gamma$ is a {\em pseudo-equilibrium} of \eqref{eq:hfGen} if $f^{St}(x;\mu) = 0$.
It is {\em admissible} if it belongs to a sticking region,
and {\em virtual} if it belongs to a detaching region.
\label{df:hPsEq}
\end{definition}

Now suppose \eqref{eq:hfGen} has a BEB
caused by a regular equilibrium hitting $\Sigma$ at $x = 0$ when $\mu = 0$.
By replacing $f(x;\mu)$ with its linearisation about $(x;\mu) = (0;0)$, and $h(x;\mu)$ with its value at $(x;\mu) = (0;0)$,
we obtain the truncated form
\begin{equation}
\begin{split}
\dot{x} &= A x + b \mu, \quad \text{for $x_1 < 0$}, \\
x &\mapsto x + v(x;\mu) c, \quad \text{whenever $x_1 = 0$}.
\end{split}
\label{eq:hf}
\end{equation}
where $A \in \mathbb{R}^{n \times n}$ and $b, c \in \mathbb{R}^n$ with $c_1 = 0$; also
\begin{equation}
v(x;\mu) = e_1^{\sf T} (A x + b \mu).
\label{eq:hv}
\end{equation}
The regular equilibrium $x^L(\mu)$ is given by \eqref{eq:fxL}, assuming $\det(A) \ne 0$.
Note that some authors prefer an equivalent truncated form
for which the regular equilibrium is fixed at origin \cite{TaCh23}.

Let $x' = x + v(x;\mu) c$ denote image of the reset law.
Then
\begin{equation}
v(x';\mu) = \left( 1 + e_1^{\sf T} A c \right) v(x;\mu),
\nonumber
\end{equation}
and hence the condition
\begin{equation}
e_1^{\sf T} A c < -1
\label{eq:hCond}
\end{equation}
ensures the reset law maps the incoming set to the outgoing set.
If $q^{\sf T} c \ne 0$ then \eqref{eq:hf}
has a unique pseudo-equilibrium $x^{St}(\mu)$ with $x^{St}(0) = 0$ \cite{DiNo08}.

\begin{theorem}
If $\det(A) \ne 0$, $e_1^{\sf T} A c < -1$, $q^{\sf T} c \ne 0$, and $x^L(\mu)$ and $x^{St}(\mu)$ are virtual,
then every solution to \eqref{eq:hf} diverges as $t \to \infty$.
\label{th:h}
\end{theorem}

Notice $f^{St}(x;\mu) = f(x;\mu) - \frac{a(x;\mu) h(x;\mu)}{\nabla v(x;\mu)^{\sf T} h(x;\mu)}$,
so at any point $x \in \Gamma$ for which $a(x;\mu) = 0$ we have $f^{St}(x;\mu) = f(x;\mu)$.
This continuity underpins the following proof.
The proof is similar to that of the previous section
except we also need to consider the action of the reset law relative to the direction $q$ (see \eqref{eq:hFinal3}).

\begin{proof}
We first characterise the admissibility of $x^{St}(\mu)$.
Observe $h(x;\mu) = c$ and $\nabla v(x;\mu)^{\sf T} = e_1^{\sf T} A$,
so the sticking vector field \eqref{eq:hfStGen} can be written as
\begin{equation}
f^{St}(x;\mu) = \frac{1}{e_1^{\sf T} A c} \left( e_1^{\sf T} A c f(x;\mu) - a(x;\mu) c \right).
\label{eq:hfStProof}
\end{equation}
Since $x^{St}(\mu)$ is a zero of $f^{St}$,
\begin{equation}
0 = e_1^{\sf T} A c f \big( x^{St}(\mu);\mu \big) - a \big( x^{St}(\mu);\mu \big) c.
\label{eq:hfSt}
\end{equation}
By multiplying both sides of this equation by
$q^{\sf T} = e_1^{\sf T} {\rm adj}(A)$, then solving for $a$, we obtain
\begin{equation}
a \big( x^{St}(\mu);\mu \big) = \frac{e_1^{\sf T} A c \,q^{\sf T} f \left( x^{St}(\mu);\mu \right)}{q^{\sf T} c}.
\nonumber
\end{equation}
But $x^{St}_1(\mu) = 0$, so
\begin{align*}
q^{\sf T} f \left( x^{St}(\mu);\mu \right)
&= q^{\sf T} (A x^{St}(\mu) + b \mu) \\
&= \det(A) x^{St}_1(\mu) + q^{\sf T} b \mu \\
&= s,
\end{align*}
where we again let $s = q^{\sf T} b \mu$, hence
\begin{equation}
a \big( x^{St}(\mu);\mu \big) = \frac{e_1^{\sf T} A c s}{q^{\sf T} c}.
\label{eq:hAdm}
\end{equation}

The first component of $x^L(\mu)$ is given by \eqref{eq:fxL1},
so we require $s \ne 0$ because $x^L(\mu)$ is virtual.
Without loss of generality suppose $s > 0$; then $\det(A) < 0$.
Also $x^{St}(\mu)$ is virtual, meaning $a \left( x^{St}(\mu);\mu \right) < 0$,
thus by \eqref{eq:hAdm} we have $q^{\sf T} c > 0$ (notice $e_1^{\sf T} A c < 0$ by the theorem statement).

For any $x \in \mathbb{R}^n$ with $x_1 \le 0$,
\begin{equation}
q^{\sf T} f(x;\mu) \ge s,
\label{eq:hFinal1}
\end{equation}
as in the proof of Theorem \ref{th:c}.
By \eqref{eq:hfStProof}, for any $x \in \Gamma$
\begin{equation}
q^{\sf T} f^{St}(x;\mu) = q^{\sf T} f(x;\mu) - \frac{a(x;\mu) q^{\sf T} c}{e_1^{\sf T} A c}.
\nonumber
\end{equation}
So since $q^{\sf T} c > 0$ and $e_1^{\sf T} A c < 0$,
for any $x \in \Gamma$ with $a(x;\mu) \ge 0$, i.e.~not in a detaching region, we have
\begin{equation}
q^{\sf T} f^{St}(x;\mu) \ge s,
\label{eq:hFinal2}
\end{equation}
using also \eqref{eq:hFinal1}.
Finally, for any $x \in \Sigma_{\rm in}$,
let $x' = x + v(x;\mu) c$ and observe
\begin{equation}
q^{\sf T}(x' - x) = v(x;\mu) q^{\sf T} c > 0,
\label{eq:hFinal3}
\end{equation}
because $v(x;\mu) > 0$.

By \eqref{eq:hFinal1}, \eqref{eq:hFinal2}, and \eqref{eq:hFinal3}
any solution $\phi(t)$ to \eqref{eq:hf} satisfies $\frac{d}{d t} q^{\sf T} \phi(t) \ge s > 0$
while in $x_1 < 0$ and on sticking regions,
plus whenever the reset law is applied the value of $q^{\sf T} \phi(t)$ increases,
hence $q^{\sf T} \phi(t) \to \infty$ as $t \to \infty$, so the solution diverges.
\end{proof}

\section{Discussion}
\label{sec:conc}

This paper has considered previously established truncated forms for BEBs and border-collision bifurcations
and shown that if two equilibria or fixed points are virtual then all orbits diverge.
It remains to determine if the same result holds for general systems,\
formulated as Conjecture \ref{cj:hots} in the case of the maps.
Implications to tipping points were discussed in \S\ref{sec:intro}.
We finish by discussing implications to bifurcation structures of truncated forms.

The `robust chaos' paper of Banerjee {\em et al.}~\cite{BaYo98}
considered a family of two-dimensional piecewise-linear maps of the form \eqref{eq:mg}
having $p^{\sf T} b \mu = 1$ fixed.
Their main result was an identification of a parameter region
where the fixed points $x^L(\mu)$ and $x^R(\mu)$ are admissible saddles and the map has a chaotic attractor.
One boundary of this region is $\det(I - A_L) = 0$
where $x^L(\mu)$ changes from admissible to virtual by \eqref{eq:mxL1}.
However, the attractor is unrelated to $x^L(\mu)$ so in fact it
persists beyond this boundary \cite{Si23e}.

Similarly crossing $\det(I - A_R) = 0$ rarely affects invariant sets other than $x^R(\mu)$.
But once $\det(I - A_L) = 0$ and $\det(I - A_R) = 0$ have both been crossed
Theorem \ref{th:m} shows that the map can no longer have an attractor, or in fact any bounded invariant set.
In this way the boundaries $\det(I - A_L) = 0$ and $\det(I - A_R) = 0$ {\em together}
have a {\em global} effect on the dynamics.
To illustrate this, Fig.~\ref{fig:trivTongues} shows how the attractor
of the three-dimensional border-collision normal form
varies with two parameters (white represents the absence of an attractor).
This form consists of \eqref{eq:mg} with
\begin{equation}
\begin{split}
A_L &= \begin{bmatrix} \tau_L & 1 & 0 \\ -\sigma_L & 0 & 1 \\ \delta_L & 0 & 0 \end{bmatrix}, \\
A_R &= \begin{bmatrix} \tau_R & 1 & 0 \\ -\sigma_R & 0 & 1 \\ \delta_R & 0 & 0 \end{bmatrix}, \\
b &= \begin{bmatrix} 1 \\ 0 \\ 0 \end{bmatrix}.
\end{split}
\label{eq:3dbcnfALARb}
\end{equation}
Notice an attractor exists for parameter values arbitrarily close to the corner
of the top-left quadrant where no attractor exists by Theorem \ref{th:m}.
Assuming Fig.~\ref{fig:trivTongues} is a typical two-dimensional slice of six-dimensional parameter space,
as appears to be the case from numerical explorations,
we can make the following conclusion.
Individually the boundaries $\det(I - A_L) = 0$ and $\det(I - A_R) = 0$ do not affect the attractor,
but together they form a codimension-two surface where the attractor is destroyed.

\begin{figure}[t!]
\begin{center}
\includegraphics[width=10cm]{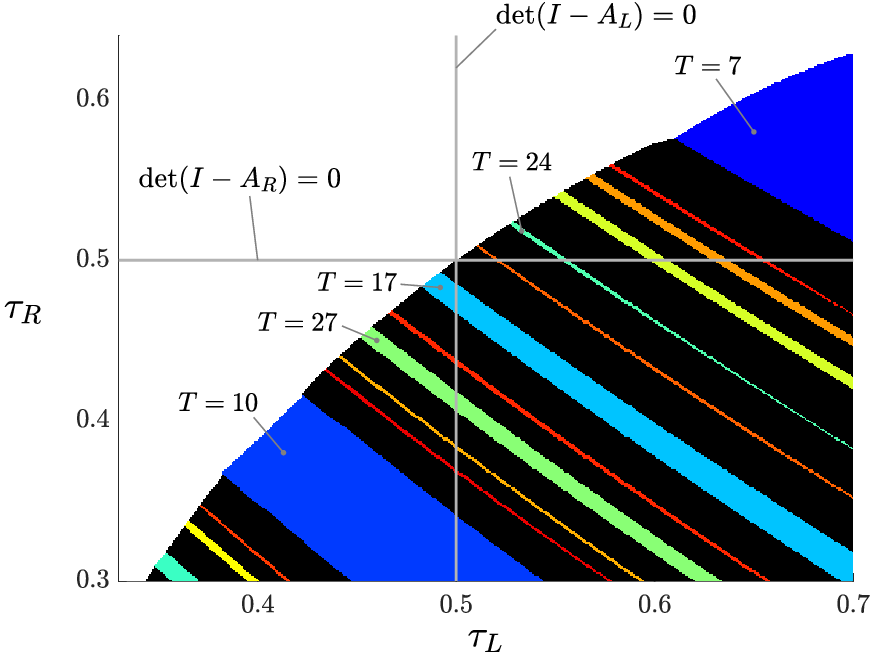}
\caption{
A two-parameter bifurcation diagram of \eqref{eq:mg} with \eqref{eq:3dbcnfALARb}
and $\sigma_L = 0$, $\delta_L = 0.5$, $\sigma_R = 1$, and $\delta_R = 1.5$.
This diagram was produced by computing $10^6$ iterates of the forward orbit of $x = 0$
over a $500 \times 500$ equi-spaced grid of $(\tau_L,\tau_R)$ values
and using the last few iterates to predict the existence of an attractor and its period.
The numerics identified a stable period-$T$ solution with $T \le 50$ in the coloured regions,
a higher period or aperiodic solution in the black regions,
and no attractor in the white regions.
\label{fig:trivTongues}
} 
\end{center}
\end{figure}

\section*{Acknowledgements}

This work was supported by Marsden Fund contract MAU2209 managed by Royal Society Te Ap\={a}rangi.

\appendix

\section{Details of the higher order terms example}
\label{app:hots}

Here we prove Conjecture \ref{cj:hots} for the example map \eqref{eq:mapExampleHOTs}.
Specifically we show the following.

\begin{proposition}
Consider the map \eqref{eq:mapExampleHOTs} with $\delta_L > 0$, $\delta_R < 0$, and $\alpha > 0$.
Then there exists $\eta > 0$ such that for any $\mu > 0$ and
$x \in B_\eta$ there exists positive $m \le \frac{6 \eta}{\mu}$ such that $\| f^m(x;\mu) \| > \eta$.
\label{pr:exampleHOTs}
\end{proposition}

\begin{proof}
Let $\eta = {\rm min} \left( \frac{\alpha}{2 \delta_L^2}, \frac{\alpha}{2 \delta_R^2}, \frac{1}{\sqrt{2}} \right)$
and $\Phi(x) = x_1 + x_2 - 2 x_2^2$.
Notice $-4 \eta \le \Phi(x) \le 2 \eta$ for any $x \in B_\eta$,
so over $B_\eta$ the value of $\Phi(x)$ ranges by at most $6 \eta$.
Using \eqref{eq:mapExampleHOTs} direct calculations give
\begin{equation}
\Phi(f(x;\mu)) - \Phi(x) = \begin{cases}
\mu - \left( \alpha + 2 \delta_L^2 x_1 \right) x_1 - 4 \delta_L x_1 x_2^2 + \left( 1 - 2 x_2^2 \right) x_2^2 \,, & x_1 \le 0, \\
\mu + \left( \alpha - 2 \delta_R^2 x_1 \right) x_1 - 4 \delta_R x_1 x_2^2 + \left( 1 - 2 x_2^2 \right) x_2^2 \,, & x_1 \ge 0.
\end{cases}
\label{eq:phiDiff}
\end{equation}
Notice $\Phi(f(x;\mu)) - \Phi(x) \ge \mu$ for any $\mu > 0$ and $x \in B_\eta$,
using $\delta_L > 0$, $\delta_R < 0$, and $\alpha > 0$.
So in $B_\eta$ the value of $\Phi$ increases by at least $\mu$ under every iteration of $f$,
thus forward orbits of $f$ escape $B_\eta$ within at most $\frac{6 \eta}{\mu}$ iterations.
\end{proof}


\end{document}